\documentclass[reqno,11pt]{amsart}
\usepackage[utf8]{inputenc}
\usepackage[T1]{fontenc}
\usepackage{geometry}
\usepackage{graphicx} % Required for inserting images
\usepackage{amsmath}
\usepackage{mathtools}
\usepackage{amsthm}
\usepackage{amssymb}
\usepackage{hyperref}
\usepackage{stmaryrd}

\newcommand{\norm}[1]{\left\lVert#1\right\rVert}

\usepackage{amsfonts}
\usepackage{xcolor}

\keywords{Fisher information, relative entropy, graphons, non exchangeable particle systems, mean field interaction}

\subjclass[2020]{65C35; 35K55; 65C05; 82C20; 26D10; 60E15}

\author{Jules~Grass}
\address{Université Claude Bernard Lyon 1, CNRS UMR 5208, Institut Camille Jordan, 69622 Villeurbanne, France. \url{grass@math.univ-lyon1.fr}}

\title{Quantitative large population limit for non exchangeable diffusions in Fisher information}
\newtheorem{theorem}{Theorem}[section]
\newtheorem{lemma}[theorem]{Lemma}

\newtheorem{hyp}[theorem]{Hypothesis}

\newtheorem{rem}[theorem]{Remark}

\begin{document}
\maketitle 
\begin{abstract}
This paper builds upon the methods developed in \cite{lacker_graphe} and \cite{fisher} to investigate the large population behavior of  non exchangeable systems of $N$ diffusive particles when the interaction matrix converges (in some sense) to a graphon. We first prove that the particle system is well approximated in Fisher information by the so-called independent projection system by proving quantitative bounds on the relative Fisher information between the marginal laws of both systems. We then use a convenient equivalence between the independent projection system and a graphon mean field system to investigate its large population behavior by proving quantitative stability estimates for graphon mean field systems in both relative entropy and Fisher information.
\end{abstract}
\section{Introduction} 
In this work, we study the large population limit of a class of interacting particles on  $\mathbb{R}^{d}$ (or $\mathbb{T}^{d}$) of the form
\begin{equation}
 d X_{t}^{i,N}=\sum_{j=1}^{N} \xi_{i,j}^{N} b\left(X_{t}^{i,N}-X_{t}^{j,N} \right)  d t+ d B_{t}^{i,N} \label{eq_main},
\end{equation}
where 
\begin{enumerate}
    \item[$\bullet$] b is smooth, bounded with bounded derivatives,
    \item[$\bullet$] $(\xi_{i,j}^{N})_{1 \leq i,j\leq N}$ is an interaction matrix satisfying $\displaystyle \max_{i \in [N]} \sum_{j=1}^{N} \xi_{i,j}^{N} \leq 1$ and $\xi_{i,j}^{N} \geq 0$ for all $(i,j)\in [N]^{2}$.
\end{enumerate}
The classical mean field case, where $\xi_{i,j}^{N}=\frac{1}{N}$, dates back to the birth of statistical mechanics and the works of Boltzmann, Vlasov, McKean and other, and it has been heavily studied by the mathematical community since the middle of the 20th century (see \cite{Chaintron_1} and \cite{Chaintron_2} for a more detailed review of the literature). While this framework has been used in other fields, including biological systems, collective behaviours of animals or neural networks (see \cite{Muntean2014}, \cite{Naldi2010} or \cite{Sirignano2019}) systems of the form \eqref{eq_main} have been used to model more intricate behaviors in social science or economics, due to the inherent inhomogeneous nature of the interactions (see \cite{Contucci2008}, \cite{Collet2014}, \cite{Budhiraja2015} or \cite{Nadtochiy2019}) . Since we are interested in understanding the large population limit of \eqref{eq_main}, a key property in the mean field case is the propagation of chaos, that is, for $k$ fixed and $N$ large, if the initial distribution of $X_{0}^{1,N}, \cdots, X_{0}^{k,N}$ is close to being tensorised then the law of  $(X^{1,N}_t,\ldots X^{k,N}_t)$ is close to the law $\bar  \mu_{t}^{\otimes k}$ of $k$ independent processes $(\bar X^{1}_t,\ldots \bar X^{k}_t)$, with $\bar X^i_t$ solutions to
\begin{equation*}
    { d} \bar X^{i}_t = b*\bar \mu_t (\bar X^i_t){ d}t+ { d}B_t,
\end{equation*}
where $B$ is a Brownian motion.
The non exchangeability nature of the system and the presence of the interaction matrix can produce a richer collection of behaviours compared to the mean field case, because it takes into consideration the large $N$ behaviour of the matrix $(\xi_{i,j}^{N})_{i,j}$. In particular, even if the initial particles of the system are independent, one cannot reasonably expect the same property to hold at a given time t if the matrix is sparse, because the interaction between the particles becomes very important. On the other hand, if the matrix is very close to the uniform case (for instance, in the case of sufficiently dense Erdös-Renyi graphs \cite{BHAMIDI20192174}, \cite{Delattre2016} \cite{Oliveira2019}), we can expect the propagation of chaos property to hold and that the macroscopic behaviour of the system \eqref{eq_main} is given by the McKean-Vlasov equation, i.e
\[dX_{t}=b\ast \mu_{t}(X_{t})dt+dB_{t}, \quad \text{with} \quad \mu_{t}=\text{Law}(X_{t}). \]
More recently, the study of systems with dense interaction graphs (that is, graphs with order of $n^{2}$ edges) has seen new developments with the use of the rich theory of graphons (see \cite{lovasz2012}, \cite{lovasz2004}). A graphon is defined as a symmmetric measurable function $G:[0,1]^{2} \to [0,1]$. We endow the space of all graphons $\mathcal{G}$ with the so called cut norm 
\[\norm{G}_{\square}=\sup_{A,B \in \mathcal{B}(\mathbb{R})} \left \lvert \int_{A\times B} G(u,v) dudv \right \rvert. \]
We can represent the interaction matrix $\xi_{i,j}^{N}$ by a piecewise constant graphon $G_{N}$ such that 
\begin{align} \label{def_graphon} \forall (u,v) \in \left[\frac{i}{N}, \frac{i+1}{N} \right[ \times \left[\frac{j}{N}, \frac{j+1}{N} \right[, \ \xi_{i,j}^{N}=\frac{G_{N}(u,v)}{N}.
\end{align}
This implies a notion of convergence for $\xi_{i,j}^{N}$, that is, the convergence in cut metric of $G_{N}$ to a graphon $G$.
 Graphon mean field systems have since been introduced ( notably in \cite{bayraktar2022graphon}) to provide a natural notion of limit for system \eqref{eq_main}. Provided that $G_{N} \to G$, we can expect the system \eqref{eq_main} to be close to the system of independent particles 
\begin{align} \label{eq_graphon_main}
\forall u \in [0,1], \ dX_{u}(t)=\int_{0}^{1} \int_{\mathbb{R}^{d}} b(X_{u}(s),x) G(u,v) P_{t}^{v}(dx)dvdt+  dB_{u}(t),
\end{align}
where $P_{t}^{v}=\text{Law}(X_{u}(t))$.
Existence, uniqueness and measurability for solutions poses significant challenges because of the regularity in $u$. In particular, there cannot exist a measurable family $(B^{u})_{u \in [0,1]}$ of independent Brownian motions. The system can still be defined through the use of rich Fubini extensions, see \cite{nonlinear_graphon} for a general study of existence and uniqueness. Several methods have been devised in order to prove that \eqref{eq_main} converges to \eqref{eq_graphon_main}, such as concentration of measure (\cite{Bayraktar_2024}) or laws of large numbers (\cite{bayraktar2022graphon}, \cite{bayraktar2021graphon}, \cite{Bet_2023}).
See also \cite{Jabin_2024} for an important recent work on graphon mean field systems. Since entropy related methods have garnered a lot of attention on the mean field case since \cite{Jabin_Wang}, due to its ability to handle singular interactions or to obtain sharp propagation of chaos rate (since \cite{lacker}), we investigate in this work the rate of convergence of \eqref{eq_main} to \eqref{eq_graphon_main} in relative entropy and Fisher information. A recent work that is very relevant to our case has been done in \cite{lacker_graphe}. They adapted Lacker's seminal BBGKY method developed in \cite{lacker} and \cite{lacker_uniforme}, usually used to prove sharp rates of propagation of chaos, in the context of non exchangeable diffusions. One important feature of their work is its non asymptotic nature, because they compare system \eqref{eq_main} to the so-called independent projection 
\begin{equation} \label{ind_proj}
dY_{t}^{i,N}=\sum_{j=1}^{N} \xi_{i,j}^{N} \left \langle b(X_{t}^{i,N}-\cdot), Q_{t}^{j,N} \right \rangle dt + dB_{t}^{i,N}, \qquad Q_{t}^{i,N}=\text{Law}(Y_{t}^{j,N}),
\end{equation}
where, if $\mu$ is a measure and $f$ an integrable function with respect to $\mu$, we denote by $ \displaystyle \left \langle f, \mu \right \rangle \coloneqq \int f d\mu$. 
For the sake of simplicity, we will denote $[N]=\llbracket 1,N\rrbracket$. They managed to prove optimal bounds (under the assumptions that $b$ is bounded and that the row sums of $\xi$ are bounded) for $H_{t}^{v}\coloneqq H(P_{t}^{v,N}|Q_{t}^{v,N})$, if $v$ is a subset of $[N]$ and 
\[ P_{t}^{v,N}=\text{Law}(X_{t}^{i,N}, \  i \in v), \qquad  Q_{t}^{v,N}=\text{Law}(Y_{t}^{i,N}, \  i \in v). \]
Since it is relevant to our work let us do a quick overview of the literature on the BBGKY hierarchy method. Lacker's original work \cite{lacker} focuses on the exchangeable case $\xi_{i,j}^{N}=\frac{1}{N}$ and relies on the use of the BBGKY hierarchy to obtain an estimate on $H_{t}^{k}=H(P_{t}^{[k],N}|Q_{t}^{[k],N})$ of the form 
\begin{align*}
\frac{d}{dt} H_{t}^{k} \leq C \frac{k^{2}}{N^{2}}+C k \left(H_{t}^{k+1}-H_{t}^{k} \right).
\end{align*}
It is then possible to show that $H_{t}^{k}= \mathcal{O} \left(\frac{k^{2}}{N^{2}} \right)$ by using Gronwall's Lemma.
It was adapted to the case of non exchangeable diffusions by using a generalisation of the BBGKY hierarchy, where $P_{t}^{v,N}$ depends on $P_{t}^{v \cup \{j \},N}$ for all $j \notin v$ and by obtaining a system of differential inequalities of the form
\[\frac{d}{dt} H_{t}^{v} \leq \mathcal{A} H_{t}^{v}+C(v), \]
where $\mathcal{A}$ is the infinitesimal generator of a Markov chain and $C$ is a constant that depends on the coefficients $\xi_{i,j}^{N}$.
The BBGKY hierarchy method has been extended in several other directions, most notably in mean field theory to prove sharp rates of propagation of chaos. The authors of \cite{lacker_uniforme} proved uniform in time results, by using log-Sobolev inequalities to gain an additional term $-c H_{t}^{k}$ on the right hand side of the system of differential inequalities. Cases of singular interactions and non constant diffusion coefficients have been treated respectively in \cite{Vortex} and \cite{article_diffusion} by considering more general systems involving $I_{t}^{k}=I\left(P_{t}^{[k],N}|Q_{t}^{[k],N}\right)$ (where $I$ is the Fisher information) of the form
\begin{align*}
\frac{d}{dt} H_{t}^{k} \leq -c_{1} I_{t}^{k}+c_{2} I_{t}^{k+1}+C\frac{k^{2}}{N^{2}}+ C \ k \left(H_{t}^{k+1}-H_{t}^{k} \right),
\end{align*}
where $c_{1}>c_{2}$. The Fisher information plays an important role in kinetic theory, notably because it is linked with entropy production via de Bruijin type relations (see \cite{Vil25} for a recent survey), and investigating its large $N$ behavior proves to be challenging. A relevant work by the author was recently made in \cite{fisher}, proving propagation of chaos in Fisher information for the first time for a class of mean field systems with smooth interaction by adapting the BBGKY hierarchy method and using fine estimates on the Hessian of the logarithm of the non linear equation proved in \cite{bdd_hess} to obtain a system of differential equations of the form
\[\frac{d}{dt}I_{t}^{k}  \leq C \frac{k^{2}}{N^{2}}+Ck\left(H_{t}^{k+1}-H_{t}^{k} \right)+Ck\left(I_{t}^{k+1}-I_{t}^{k} \right)+C \ I_{t}^{k}.\]
Using the estimates of \cite{lacker} on systems of differential inequalities, they proved the bound $I_{t}^{k}=\mathcal{O}\left( \frac{k^{2}}{N^{2}} \right)$ and showed that it was optimal on a simple Gaussian example. In this work, we prove in Theorem \ref{th:main} similar bounds on Fisher's information in the context of non exchangeable diffusions. These bounds are non asymptotic and do not take into consideration any notion of a limit for the interaction matrix but they show that, if the graph is sufficiently dense, system \eqref{eq_main} is close to the independent projection.

Taking advantage of the independence, we then investigate the large population behavior of the independent projection system when the interaction converges to a graphon and use its very convenient equivalence to a graphon mean field system (in a very similar way as in \cite{Jabin_2024}). 
 Notice that, if $G_{N}$ is defined as \eqref{def_graphon} then if we define
\begin{align*}
\left\{
            \begin{array}{ll}
              dX_{u}^{N}(t)=\int_{0}^{1} \int_{\mathbb{R}^{d}} b(X_{u}^{N}(s),x) G_{N}(u,v) P_{t}^{v,N}(dx)dvdt+  dB_{u}(t),  \\ 
             P_{t}^{v,N}=\text{Law}(X_{u}^{N}(t)),
            \end{array}
            \right.
\end{align*}
the law of $X_{u}^{N}(t)$ is constant in $u \in [\frac{i}{N}, \frac{i+1}{N}[$, and is equal to the law of $Y_{t}^{i,N}$ (note that the existence of a solution to this system has been proven in \cite{nonlinear_graphon} for instance). Therefore, if we assume that $G_{N} \to G$, we can naturally expect that the independent projection will be close to the system 
\begin{align*}
\left\{
            \begin{array}{ll}
              dX_{u}(t)=\int_{0}^{1} \int_{\mathbb{R}^{d}} b(X_{u}(s),x) G(u,v) P_{t}^{v}(dx)dvdt+  dB_{u}(t), \\ 
              P_{t}^{v}=\text{Law}(X_{u}(t)).
            \end{array}
            \right.
\end{align*}
Proving the convergence then amounts to proving stability estimates for graphon mean field systems.
We prove in Theorem \ref{ent_graphon} and \ref{fisher_graphon} that we can control the relative entropy and Fisher information between the marginal laws of two graphon mean field systems (under conditions on the interaction function and initial condition) with underlying graphons $G_{1}$ and $G_{2}$ by $C d(G_{1},G_{2})^{2}$, where $d$ is a distance that appears naturally in our computations.

The plan of the paper is the following. Section \ref{sec_2} presents our main results with our bounds on the propagation of independence in Fisher information, as well as our stability estimates for graphon mean field systems in both relative entropy and Fisher information. Section \ref{sec_3} is dedicated to the proof of Theorem \ref{th:main} and Sections \ref{sec_4} and \ref{stabilite} are dedicated to the proof of our stability estimates in relative entropy and Fisher information respectively. In order to prove our results on Fisher information, we rely on both the fact that the law of the graphon mean field system is a strong solution of a PDE (proven in Section \ref{sol_forte}) and a fine control on the Hessian of the law of a graphon mean field system, which is stated and proven in Section \ref{sec_7}.

\section{Main results} \label{sec_2}
Let us first recall the definition of the relative entropy and Fisher information : for any probability measures $\mu$ and $\nu$ such that $\mu$ admits $f$ as density with respect to $\nu$, we define
\begin{equation*}
    H(\mu|\nu) = \int \nu f \log f,
\end{equation*}
and
\begin{equation*}
    I(\mu|\nu) = \int \nu f \Vert \nabla \log f\Vert^2 = \int \nu \frac{\Vert \nabla f\Vert^2}{f}.
\end{equation*}
We will make the following hypotheses for the initial conditions throughout the paper.
\begin{hyp}
\label{hyp-princ}
A couple of density of probability measures (with respect to the Lebesgue measure) $(P,Q)$ is said to satisfy Hyothesis 2.1 if
\begin{itemize}
    \item Torus Case.  $Q$ is $C^2$, bounded from below.
    \item $\mathbb{R}^d$ Case. We suppose initial Gaussian bounds. More precisely that there exist constants $c$, $C$, $c'$, $C'$  such that
    \[
    C^{-1}e^{-c^{-1}\norm{x}^2}\leq P(x)\leq Ce^{-c\norm{x}^2},\quad  C'^{-1}e^{-c'^{-1}\norm{x}^2}\leq Q(x)\leq C'e^{-c'\norm{x}^2},
    \]
    
    and that, for $j=1,2,3$,
    \[
    \norm{\nabla^j \log P}^2 \leq C\left(1+ \left|\log P\right|^j\right), \quad \norm{\nabla^j \log Q}^2 \leq C'\left(1+ \left|\log Q\right|^j\right).
    \]
\end{itemize}
\end{hyp}
Remark that in the $\mathbb{R}^d$ case, the assumption is indeed verified for $P$ and $Q$  Gaussian measures.
We may now state the first result of this paper
\begin{theorem}\label{th:main}
We let $\displaystyle \delta_{N}=\max_{(i,j) \in v^{2}} \xi_{i,j}^{N}$ and
\[H_{t}^{v}=H \left( P_{t}^{v,N}|Q_{t}^{v,N} \right), \qquad \quad I_{t}^{v}=I\left(P_{t}^{v,N}|Q_{t}^{v,N} \right).\]
Fix a $T>0$ and suppose that there exists a constant $C$ such that
\begin{equation*}
\forall v \subset [N], \
    H_{0}^{v}+I_{0}^{v} \leq C \left(\delta_{N} |v|+1\right)\left( \sum_{(i,j)\in v^{2}} \left(\xi_{i,j}^{N} \right)^{2}+\delta_{N} \sum_{(i,j) \in v} \left(\xi^{T} \xi+\xi \xi^{T} \right)_{i,j}+\delta_{N}^{2} |v| \right).
\end{equation*}
Suppose moreover that $\left(P_{0}^{[N],N}, Q_{0}^{[N],N}\right)$ satisfies Hypothesis \ref{hyp-princ}.
We distinguish two set of assumptions:
\begin{itemize}
    \item Torus Case. Suppose that $b$ is smooth and assume that that  $\nabla^2 \log Q_{0}^{u,N}$ is bounded with constants that do not depend on $u$ or $N$.
    \item $\mathbb{R}^d$ case. Suppose that $b$ is smooth bounded, with bounded derivatives and assume that there exists $C>0$ such that 
    \begin{equation*}
   \forall u \in [0,1], \ \sup_{t\in [0,T]}\Vert \nabla^2 \log Q_{t}^{u,N}\Vert \leq C.
\end{equation*}
%Suppose moreover that there exists V, W smooth such that $\Gamma= \nabla W$ and $F=\nabla V$
\end{itemize}
Then, there exists a constant $M_{T}$ (that does not depend on v,N) such that 
\begin{equation*}
\forall v \subset [N], \
    H_{T}^{v}+I_{T}^{v} \leq M_{T} \left(\delta_{N} |v|+1\right)\left( \sum_{(i,j)\in v^{2}} \left(\xi_{i,j}^{N} \right)^{2}+\delta_{N} \sum_{(i,j) \in v} \left(\xi^{T} \xi+\xi \xi^{T} \right)_{i,j}+\delta_{N}^{2} |v| \right).
\end{equation*}
In particular, if $\xi_{i,j}^{N}=G_{N} \left( \frac{i}{N}, \frac{j}{N} \right)$ with $G_{N}$ a graphon with values in $[0,1]$, we have 
\begin{equation*}
\forall v \subset [N], \
    H_{T}^{v}+I_{T}^{v} \leq M_{T} \frac{\left|v \right|^{2}}{N^{2}}.
\end{equation*}
\end{theorem}
\begin{rem}
\indent
\begin{itemize}
\item  As in \cite{fisher}, we suppose for simplicity that all derivatives are bounded, but derivatives up to four may be sufficient.
\item Our proof relies on very similar estimates as in \cite{lacker_graphe} to obtain a bound of the form 
\begin{align*}
\frac{d}{dt} I_{t}^{v} \leq C_{0} I_{t}^{v}+ \mathcal{A} H_{t}^{v}+\mathcal{A} I_{t}^{v}+C(v),
\end{align*}
where $\mathcal{A}$ is an operator defined on functions $F:\mathcal{P}([N])\to \mathbb{R}$ by
\begin{align*}
\forall v \in \mathcal{P}([N]), \ \mathcal{A}F(v)=\sum_{i \in v} \sum_{k \notin v} \xi_{i,k}^{N} \left( F(v \cup \{k \})-F(v)\right).
\end{align*}
C is a function of $v$ depending of the coefficients $\displaystyle (\xi_{i,j}^{N})_{i,j}$. In our proof, we will use 
\begin{align*} C(v)=C_{1}\sum_{i \in v } \left(\sum_{j \in v} \xi_{i,j}^{N} \right)^{2},
\end{align*}
where $C_{1}$ is a constant that does not depend on $N$ or $v$.
We prove the final bound by using a change of variable of the form $Z_{t}^{v}=\alpha H_{t}^{v}+I_{t}^{v}$ to get the estimate
\begin{align*}
\frac{d}{dt} Z_{t}^{v} \leq \mathcal{A}Z_{t}^{v}+C(v).
\end{align*}
Because we rely on the exact same bounds as in \cite{lacker_graphe}, we can also get similar results as their Theorem 2.8, 2.9, 2.11 on averaged and maximum Fisher information, with the exception of uniform in time results, because they rely on a fine analysis of the constants and other assumptions.
\item We only considered systems subject only to an interaction term for the sake of simplicity but we can also consider smooth potentials with bounded derivatives.
\item Note that, because of the equivalence of the independent projection and graphon mean field systems, Lemma \ref{lemme_hessienne} provides a framework for the second assumption of the Theorem.
\end{itemize}
\end{rem}
The proof mainly relies on the use of the BBGKY hierarchy method for non exchangeable diffusions done in \cite{lacker_graphe}  and on the bound on Fisher information of Lemma 2.4 of \cite{fisher}. Provided that the bound of Theorem \ref{th:main} vanishes as $N \to +\infty$, we therefore know that the large population behavior of a finite number of particles of system \eqref{eq_main} is the same as \eqref{ind_proj}. Due to the equivalence in law of system \eqref{ind_proj} with a graphon mean field system stated in the introduction, we investigate the behavior of \eqref{ind_proj} through stability estimates for graphon mean field systems. 
The goal of the following two Theorems is to obtain bounds on $H\left(Q_{t}^{u,N}|Q_{t}^{u} \right)$ and $I\left(Q_{t}^{u,N}|Q_{t}^{u} \right)$.
\begin{theorem} \label{ent_graphon}
Let $T>0$ and $G_{1}, G_{2}$ two bounded graphons. Consider 
For $i \in \{1,2\}$ and $u \in [0,1]$, the two associated graphon particle systems
\begin{align*}
dX_{u}^{i}(t)=\int_{0}^{1} \int_{\mathbb{R}^{d}} b(X_{u}^{i}(s),x) G_{i}(u,v) P_{t}^{i,v}(dx)dvdt+  dB_{u}(t) 
          \quad \text{where} \quad     P_{t}^{i,v}=\text{Law}(X_{u}^{i}(t)).
\end{align*}
Let, for $u \in [0,1]$: $P_{[t]}^{i,u}=\text{Law}(X_{s}^{i,u}, s \in [0,t])$ and $H_{t}^{u}=H\left(P_{[t]}^{1,u}|P_{[t]}^{2,u}\right)$. Suppose 
\begin{itemize}
    \item b is bounded,
    \item $\displaystyle \forall u \in [0,1], \ H_{0}^{u} \leq C \left(\underset{u\in [0,1]}{\sup} \int_{0}^{1} \left| G_{1}(u,v)-G_{2}(u,v) \right| dv \right)^{2}.$
\end{itemize}
Then, there exists a constant $M_{T}$ (that does not depend on $u$) such that, for all $u \in [0,1]$,
\begin{equation*}
H_{T}^{u} \leq M_{T}  \left(\underset{u\in [0,1]}{\sup} \int_{0}^{1} \lvert G_{1}(u,v)-G_{2}(u,v) \rvert dv \right)^{2}.
\end{equation*}
\end{theorem}
\begin{theorem} \label{fisher_graphon}
Let $T>0$ and $G_{1}, G_{2}$ two bounded graphons. Consider 
For $i \in \{1,2\}$ and $u \in [0,1]$, the two associated graphon particle systems
\begin{align*}
dX_{u}^{i}(t)=\int_{0}^{1} \int_{\mathbb{R}^{d}} b(X_{u}^{i}(s),x) G_{i}(u,v) P_{t}^{i,v}(dx)dvdt+  dB_{u}(t) 
          \quad \text{where} \quad     P_{t}^{i,v}=\text{Law}(X_{u}^{i}(t)).
\end{align*}
Let, for $u \in [0,1]$: $I_{t}^{u}=I\left(P_{t}^{1,u}|P_{t}^{2,u}\right)$. Suppose that $(P_{0}^{1,u},P_{0}^{2,u})$ satisfies Hypothesis \ref{hyp-princ} for all $u \in [0,1]$ and that
\begin{itemize}
    \item b is smooth
    \item $\displaystyle \forall u \in [0,1], \ I_{0}^{u}+H_{0}^{u} \leq C \left(\underset{u\in [0,1]}{\sup} \int_{0}^{1} \left| G_{1}(u,v)-G_{2}(u,v) \right| dv \right)^{2},$
    \item  
 $\displaystyle  \forall u \in [0,1], \ \sup_{t\in [0,T]}\Vert \nabla^2 \log Q_{t}^{u,N}\Vert \leq C.$

\end{itemize}
Then, there exists a constant $M_{T}$ (that does not depend on $u$) such that, for all $u \in [0,1]$,
\begin{equation*}
I_{T}^{u}+H_{T}^{u} \leq M_{T}  \left(\underset{u\in [0,1]}{\sup} \int_{0}^{1} \lvert G_{1}(u,v)-G_{2}(u,v) \rvert dv \right)^{2}.
\end{equation*}
\end{theorem}
\begin{rem}
\indent
\begin{itemize} \item While the assumption that $\displaystyle 
   \forall u \in [0,1], \ \sup_{t\in [0,T]}\Vert \nabla^2 \log Q_{t}^{u,N}\Vert$ is uniformly bounded may seem strong, Lemma \ref{lemme_hessienne} provides a framework where it is satisfied.
\item Note that the distance considered is not the usual cut metric for graphon, but it appears naturally in our case. Moreover, we have the inequality, for all graphons $G_{1}$ and $G_{2}$
\[\norm{G_{1}-G_{2}}_{\square}\leq \underset{u\in [0,1]}{\sup} \int_{0}^{1} \left| G_{1}(u,v)-G_{2}(u,v) \right| dv.\]
\end{itemize}
\end{rem}
In order to prove Theorem \ref{fisher_graphon} and \ref{ent_graphon}, we will rely on the PDEs satisfied by both $P_{t}^{1,u}$ and $P_{t}^{2,u}$. Since we use the fact that both densities are strong solutions to justify our computations, we dedicate Section \ref{sol_forte} to the proof of that fact.

\section{Proof of Theorem 2.2} \label{sec_3}

We first start by computing the time evolution of the relative Fisher information by using a result from \cite{fisher}, then we bound all the terms and obtain a system of differential inequalities. Using the estimates of \cite{lacker_graphe} yields the result.
Notice that we have, for $v \subset [N]$ with $v \neq [N]$:
\begin{align*}
\partial_{t} Q_{t}^{v}=  \Delta Q_{t}^{v}-\sum_{i \in v} \nabla_{x_{i}} \cdot \left(b_{1}^{i,v} Q_{t}^{v} \right),
\end{align*}
\begin{align*}
\partial_{t} P_{t}^{v}=  \Delta P_{t}^{v}-\sum_{i \in v} \nabla_{x_{i}} \cdot \left(b_{2}^{i,v} P_{t}^{v} \right),
\end{align*}
with $b_{1}^{i,v}$ and $b_{2}^{i, [n]}$ being smooth, bounded functions with bounded derivatives defined as 
\begin{align*}
& b_{1}^{i,v}=\sum_{j=1}^{N} \xi_{i,j}^{N} \left \langle b(x_{i}-\cdot), Q_{t}^{j} \right \rangle 
\quad \text{and} \quad 
b_{2}^{i,v}=\sum_{j \in v} \xi_{i,j}^{N} b(x_{i}-x_{j})+\sum_{j \notin v} \xi_{i,j}^{N} \left \langle b(x_{i}-\cdot), P_{t}^{v \cup \{j\}|v} \right \rangle.
\end{align*}
Since the bound on $H_{t}^{v}$ has already been done in \cite{lacker_graphe}, we will only focus on $I_{t}^{v}$. In order to so, notice that we can use Lemma 2.4 (and more precisely the second point of remark 2.5) of \cite{fisher}. While it is not straightforward to check its assumptions, our hypothesis are the same so we can prove very similar regularity estimates to those of their Section 6 for $P_{t}^{v,N}$ and $Q_{t}^{v,N}$ to justify all the integration by parts we use. Let $\varepsilon >0$. According to this Lemma, the evolution of $I_{t}^{v}$ is bounded by
\begin{align*}
\frac{d}{dt} I_{t}^{v} &\leq -(2-\varepsilon)\sum_{i,j \in v}  \int P_{t}^{v} \left\Vert \nabla^2_{v_i,v_j} \log \frac{P_{t}^{v}}{Q_{t}^{v}}\right\Vert^2  \\
&\quad +4\sum_{i,j \in v} \int P_{t}^{v} \ \nabla_{v_i} \log \frac{P_{t}^{v}}{Q_{t}^{v}} \cdot \nabla^2_{v_i,v_j} \log Q_{t}^{v} \nabla_{v_j} \log \frac{P_{t}^{v}}{Q_{t}^{v}}
\\&
\quad -2 \sum_{i,j \in v} \int P_{t}^{v} \ \nabla_{v_{i}} \log \frac{P_{t}^{v}}{Q_{t}^{v}} \cdot \nabla_{v_{i}} b_{2}^{j} \ \nabla_{v_{j}} \log \frac{P_{t}^{v}}{Q_{t}^{v}}\\
&
\quad +2  \sum_{i,j \in v} \int P_{t}^{v} \ \nabla_{v_{i}} \log \frac{P_{t}^{v}}{Q_{t}^{v}} \cdot \nabla^2_{v_{i},v_{j}} \log Q_{t}^{v}  \left(b_{2}^{j}-b_{1}^{j}\right)
\\ &\quad + C\sum_{i,j \in v} \int P_{t}^{v} \norm{\nabla_{v_{i}} \left(b_{2}^{j}-b_{1}^{j} \right)}^{2}.
\end{align*}
Notice moreover that $\nabla^{2} \log Q_{t}^{v}$ is uniformly bounded by a constant that does not depend in N or v according to Lemma \ref{lemme_hessienne}. Moreover, because the particles at the limit are independent, we have that, if $i \neq j$, $\nabla^{2}_{v_{i}, v_{j}} \log Q_{t}^{v}=0$ so that 
\begin{align*}
 \sum_{i,j \in v} \int P_{t}^{v} \ \nabla_{v_i} \log \frac{P_{t}^{v}}{Q_{t}^{v}} \cdot \nabla^2_{v_i,v_j} \log Q_{t}^{v} \ \nabla_{v_j} \log \frac{P_{t}^{v}}{Q_{t}^{v}} \leq C I_{t}^{v},
\end{align*}
with $C$ constant that does not depend on $v$ or $N$. Therefore, by using the fact that $\nabla_{v_{i}} b_{2}^{j}$ is bounded and the inequality $x \cdot y \leq \norm{x}^{2}+\frac{1}{4 } \norm{y}^{2}$ in the fourth term and taking $\varepsilon$ small enough, we have 
\begin{align*}
\frac{d}{dt} I_{t}^{v} \leq C I_{t}^{v}+C \underbrace{ \sum_{i \in v} \int P_{t}^{v} \norm{b_{2}^{i,v}-b_{1}^{i,v}}^{2}}_{A}+C \underbrace{ \sum_{(i,j) \in v^{2}} \int P_{t}^{v} \norm{\nabla_{x_{j}} b_{2}^{i,v}-\nabla_{x_{j}} b_{1}^{i,v}}^{2} }_{ B}.
\end{align*}
We can bound A by using very similar arguments as in \cite{lacker}, while we will bound B by using \cite{fisher}. We have, by replacing $b_{1}^{i,v}$ and $b_{2}^{i,v}$ by their expression and by using the fact that $\sum_{i=1}^{N} \xi_{i,j}^{N} \leq 1$ and that $b$ is bounded,
\begin{align*}
A& =\sum_{i \in v} \int P_{t}^{v} \norm{\sum_{k \in v} \xi_{i,k}^{N} \left( b(x_{i}-x_{k})-\left \langle b(x_{i}-\cdot), Q_{t}^{k} \right \rangle \right) +\sum_{k \notin v} \xi_{i,k}^{N} \left \langle b(x_{i}-\cdot), P_{t}^{\{k\} \cup v|v}-Q_{t}^{k} \right \rangle }^{2} 
\\ &
\leq C \sum_{i \in v} \int P_{t}^{v} \sum_{k \in v} \xi_{i,k}^{N} \sum_{k \in v} \xi_{i,k}^{N} \norm{b(x_{i}-x_{k})-\left \langle b(x_{i}-\cdot),Q_{t}^{k} \right \rangle}^{2}
\\ & \quad+\sum_{i \in v} \int P_{t}^{v} \sum_{k \notin v} \xi_{i,k}^{N} \sum_{k \notin v} \xi_{i,k}^{N} \norm{\left \langle b(x_{i}-\cdot), P_{t}^{\{k\} \cup v|v}-Q_{t}^{k} \right \rangle }^{2}
\\ &
\leq C \sum_{i \in v} \left(\sum_{k \in v} \xi_{i,k}^{N} \right)^{2}+C \sum_{i \in v} \sum_{k \notin v} \xi_{i,k}^{N} \int P_{t}^{v} \norm{\left \langle b(x_{i}-\cdot), P_{t}^{\{k\} \cup v|v}-Q_{t}^{k} \right \rangle }^{2}.
\end{align*}
By using Pinsker's inequality and the fact that $b$ is bounded as well as a towering property of the relative entropy, we have
\begin{align*}
A \leq C(v)+C \ \mathcal{A} H_{t}^{v},
\end{align*}
where, for every $v \subset [n]$ and $F:  \mathcal{P}([n]) \to \mathbb{R}$,
\begin{align*}
C(v) \coloneqq C \sum_{i \in v} \left(\sum_{k \in v} \xi_{i,k}^{N} \right)^{2}, \qquad \mathcal{A} F(v)=\sum_{i \in v} \sum_{k \notin v} \xi_{i,k}^{N} \left( F(v \cup \{k \})-F(v) \right).
\end{align*}
Let us now bound $B$. Notice that we have, if $(i,j) \in v^{2}$ and $i=j$,
\begin{align*}
\nabla_{x_{i}} \left(b_{2}^{i,v}-b_{1}^{i,v} \right) & = \sum_{\substack{k \in v \\ k \neq i}} \xi_{i,k}^{N} \left( \nabla b(x_{i}-x_{k})-\left \langle \nabla b(x_{i}-\cdot), Q_{t}^{k} \right \rangle \right)
\\ & +\sum_{\substack{k \notin v \\ k \neq i}} \xi_{i,k}^{N} \left \langle \nabla b(x_{i}-\cdot), P_{t}^{v \cup \{k \}|v}-Q_{t}^{k} \right \rangle + \sum_{\substack{k \notin v \\ k \neq i}} \xi_{i,k}^{N} \left \langle b(x_{i}-\cdot), \nabla_{x_{i}} P_{t}^{v \cup \{k \}|v} \right \rangle.
\end{align*}
 Moreover, if $i \neq j$, we have
 \begin{align*}
 \nabla_{x_{j}} \left(b_{2}^{i,v}-b_{1}^{i,v} \right) &=\nabla_{x_{j}} \left(\sum_{k \in v} \xi_{i,k}^{N} b(x_{i}-x_{k})+\sum_{k \notin v} \xi_{i,k}^{N} \left \langle b(x_{i}-\cdot), P_{t}^{v \cup \{k\}|v} \right \rangle  \right)
 \\ &= -\xi_{i,j}^{N} \nabla b(x_{i}-x_{j})+ \sum_{k \notin v} \xi_{i,k}^{N} \left \langle b(x_{i}-\cdot), \nabla_{x_{j}} P_{t}^{v \cup \{k\}|v} \right \rangle .
 \end{align*}
Therefore, we have
\begin{align*}
 B & \leq C \sum_{i \in v} \int P_{t}^{v} \norm{\nabla_{x_{i}} \left(b_{2}^{i,v}-b_{1}^{i,v} \right)}^{2}+C \sum_{\substack{(i,j) \in v^{2} \\ i \neq j}} \int P_{t}^{v} \norm{\nabla_{x_{j}} \left(b_{2}^{i,v}-b_{1}^{i,v} \right)}^{2}
 \\ &
 \leq C \sum_{i \in v} \int P_{t}^{v} \left \lVert \sum_{\substack{j \in v \\ k \neq i}} \xi_{i,k}^{N} \left( \nabla b(x_{i}-x_{k})-\left \langle \nabla b(x_{i}-\cdot), Q_{t}^{k} \right \rangle \right) \right.
 \\ &
\qquad \qquad \qquad \qquad \qquad \qquad \left. +\sum_{\substack{k \notin v \\ k \neq i}} \xi_{i,k}^{N} \left \langle \nabla b(x_{i}-\cdot), P_{t}^{v \cup \{k \}|v}-Q_{t}^{k} \right \rangle \right \rVert^{2}
 \\ &
 \quad+C \sum_{(i,j) \in v} \int P_{t}^{v} \norm{-\xi_{i,j}^{N} \nabla b(x_{i}-x_{j})+ \sum_{k \notin v} \xi_{i,k}^{N} \left \langle b(x_{i}-\cdot), \nabla_{x_{j}} P_{t}^{v \cup \{k\}|v} \right \rangle }^{2}
 \\ & 
 \quad +C \sum_{i \in v} \int P_{t}^{v} \norm{ \sum_{\substack{k \notin v \\ k \neq i}} \xi_{i,k}^{N} \left \langle b(x_{i}-\cdot), \nabla_{x_{i}} P_{t}^{v \cup \{k \}|v} \right \rangle}^{2}.
\end{align*}
Notice that, by using very similar arguments that we used to bound A, the first term in the last inequality is bounded by $C(v)+C \mathcal{A}H_{t}^{v}$.
By using the inequality $\norm{x+y}^{2} \leq 2 \left(\norm{x}^{2}+\norm{y}^{2} \right)$ we have 
\begin{align*}
& \sum_{(i,j) \in v}  \int P_{t}^{v}  \norm{-\xi_{i,j}^{N} \nabla b(x_{i}-x_{j})+ \sum_{k \notin v} \xi_{i,k}^{N}  \left \langle b(x_{i}-\cdot), \nabla_{x_{j}} P_{t}^{v \cup \{k\}|v} \right \rangle }^{2} \\ & \leq C \sum_{(i,j)\in v^{2}} \int P_{t}^{v} \left(\xi_{i,j}^{N} \norm{\nabla b(x_{i}-x_{j})} \right)^{2}
 +C \sum_{(i,j) \in v^{2}} \int P_{t}^{v} \norm{\sum_{k \notin v} \xi_{i,k}^{N} \left \langle b(x_{i}-\cdot), \nabla_{x_{i}} P_{t}^{v \cup \{k \}|v} \right \rangle}^{2}.
\end{align*}
Because $\nabla b$ is bounded and the $(\xi_{i,j}^{N})_{i,j}$ are nonnegative, we can bound the first term
\begin{align*}
\sum_{(i,j)\in v^{2}} \int P_{t}^{v} \left(\xi_{i,j}^{N} \norm{\nabla b(x_{i}-x_{j})} \right)^{2} & \leq C \sum_{(i,j) \in v^{2}} \left(\xi_{i,j}^{N} \right)^{2}
\\ &
\leq \sum_{i \in v} \left( \sum_{j \in v} \xi_{i,j}^{N} \right)^{2}
\\ &
\leq C(v).
\end{align*}
For the second term, notice that we have, by the Cauchy-Schwarz inequality and the fact that $\sum_{i} \xi_{i,j}^{N} \leq 1$
\begin{align*}
\sum_{(i,j) \in v^{2}} \int P_{t}^{v} & \norm{\sum_{k \notin v} \xi_{i,k}^{N} \left \langle b(x_{i}-\cdot), \nabla_{x_{i}} P_{t}^{v \cup \{k \}|v} \right \rangle}^{2} \\ & \leq \sum_{(i,j) \in v^{2}} \int P_{t}^{v} \left(\sum_{k \notin v} \xi_{i,k}^{N} \right) \sum_{k \notin v} \xi_{i,k}^{N} \norm{\left \langle b(x_{i}-\cdot), \nabla_{x_{j}} P_{t}^{v \cup \{k \}|v} \right \rangle}^{2}
\\ &
\leq  \sum_{(i,j) \in v^{2}} \sum_{k \notin v} \xi_{i,k}^{N} \int P_{t}^{v} \norm{\left \langle b(x_{i}-\cdot), \nabla_{x_{j}} P_{t}^{v \cup \{k \}|v} \right \rangle}^{2}
\\ &
\leq \sum_{(i,j) \in v^{2}} \sum_{k \notin v} \xi_{i,k}^{N} \int P_{t}^{v} \norm{\int P_{t}^{v \cup \{k \}|v}(x^{\star}) b(x_{i}-x^{\star}) \cdot \nabla_{x_{j}}  \log P_{t}^{v \cup\{k \}|v}}^{2}
\\ &
\leq \sum_{(i,j) \in v^{2}} \sum_{k \notin v} \xi_{i,k}^{N} \int P_{t}^{v} \int P_{t}^{v \cup \{k \}|v}(x^{\star}) \norm{b(x_{i}-x^{\star}) \cdot \nabla_{x_{j}} \log  P_{t}^{v \cup\{k \}|v}}^{2}.
\end{align*}
Notice that we have the following property of decomposition of Fisher information, notably used in \cite{fisher}
\begin{align*}
I_{t}^{v \cup \{k\}}-I_{t}^{v}& =\sum_{j \in v} \int P_{t}^{v \cup \{k\}} \norm{\nabla_{x_{j}} \log P_{t}^{v \cup \{k\}|v}}^{2}+\int P_{t}^{v \cup \{k\}} \norm{\log \frac{P_{t}^{v \cup \{k\}}}{Q_{t}^{v \cup \{k\}}}}^{2}
\\ &
\geq \sum_{j \in v} \int P_{t}^{v \cup \{k\}} \norm{\nabla_{x_{j}} \log P_{t}^{v \cup \{k\}|v}}^{2}.
\end{align*}
We get, by using the previous identity and the fact that $b$ is bounded 
\begin{align*}
\sum_{(i,j) \in v^{2}} \int P_{t}^{v} \norm{\sum_{\substack{k \notin v }} \xi_{i,k}^{N} \left \langle b(x_{i}-\cdot), \nabla_{x_{j}} P_{t}^{v \cup \{k \}|v} \right \rangle}^{2} & \leq C \sum_{(i,j) \in v^{2}} \sum_{\substack{k \notin v}} \xi_{i,k}^{N} \int P_{t}^{v \cup \{k \}} \norm{\nabla_{x_{j}} P_{t}^{v \cup \{k \}|v}}^{2}
\\ &
\leq C \sum_{i \in v} \sum_{k \notin v} \xi_{i,k}^{N} \left(I_{t}^{v \cup \{k \}}-I_{t}^{v} \right)
\\ &
\leq C \mathcal{A} I_{t}^{v}.
\end{align*}
For the last term, we can use a very similar argument than before, by using the fact that $b$ is bounded
\begin{align*}
 \sum_{i \in v} \int P_{t}^{v} & \norm{ \sum_{k \notin v } \xi_{i,k}^{N} \left \langle b(x_{i}-\cdot), \nabla_{x_{i}} P_{t}^{v \cup \{k \}|v} \right \rangle}^{2} \\ & \quad \quad \leq \sum_{i \in v} \sum_{k \notin v} \xi_{i,k}^{N} \int P_{t}^{v} \norm{\left \langle b(x_{i}-\cdot), \nabla_{x_{i}} P_{t}^{v \cup \{k \}|v} \right \rangle}^{2}
 \\ & \quad \quad
 \leq C \sum_{i \in v} \sum_{k \notin v} \xi_{i,k}^{N} \int P_{t}^{v} \norm{\int P_{t}^{v \cup \{k \}|v}(x*) b(x_{i}-x*) \nabla_{x_{i}} \log P_{t}^{v \cup \{k\}|v}}^{2}
 \\ & \quad \quad
 \leq C \sum_{i \in v} \sum_{k \notin v} \xi_{i,k}^{N} \int P_{t}^{v \cup \{k\}} \norm{\nabla_{x_{i}} \log P_{t}^{v \cup \{k \}|v}}^{2}
 \\ & \quad \quad
 \leq C \sum_{i \in v} \sum_{k \notin v} \xi_{i,k}^{N} \left(I_{t}^{v\cup \{k\}}-I_{t}^{v} \right).
\end{align*}
Finally, we have
\begin{align*}
\frac{d}{dt} I_{t}^{v} \leq C I_{t}^{v}+C(v)+C \mathcal{A} H_{t}^{v}+C \mathcal{A} I_{t}^{v}.
\end{align*}
We also know that $H$ satisfies (from \cite{lacker_graphe})
\begin{align*}
\frac{d}{dt} H_{t}^{v} \leq -\frac{1}{2} I_{t}^{v}+C(v)+C \mathcal{A} H_{t}^{v}.
\end{align*}
Therefore, for all $\alpha>0$, we have, if $Z_{t}^{v} \coloneqq I_{t}^{v}+\alpha H_{t}^{v}$
\begin{align*}
\frac{d}{dt} Z_{t}^{v} \leq \left(C-\frac{\alpha}{2}\right)I_{t}^{v}+C \mathcal{A} Z_{t}^{v}+C(v).
\end{align*}
Using Theorem 2.15 from \cite{lacker_graphe}, we get, if $\displaystyle \delta_{N}=\max_{(i,j) \in v^{2}} \xi_{i,j}^{N}$
\begin{align*}
I_{t}^{v} \leq \left(\delta_{N} |v|+1\right)\left( \sum_{(i,j)\in v^{2}} \left(\xi_{i,j}^{N} \right)^{2}+\delta_{N} \sum_{(i,j) \in v} \left(\xi^{T} \xi+\xi \xi^{T} \right)_{i,j}+\delta_{N}^{2} |v| \right).
\end{align*}
Note finally that, if $\xi_{i,j}^{N}=\frac{G_{N} \left(\frac{i}{N}, \frac{j}{N} \right)}{N}$, with $G_{N}$  graphon with values in $[0,1]$, we have
\begin{align*}
I_{t}^{v} \leq \frac{|v|^{2}}{N^{2}}.
\end{align*}

\section{Stability of graphon mean field systems} \label{sec_4}

 In this section, we let $u \in [0,1]$. We let $P_{[t]}^{1,u}=\text{Law}(X_{s}^{1,u}, s \in [0,t])$ and $P_{[t]}^{2,u}=\text{Law}(X_{s}^{2,u}, s \in [0,t])$. We will control the relative entropy $H_{t}^{u}$ between $P_{[t]}^{1,u} $ to $P_{[t]}^{2,u}$ in terms of a distance between $G_{1}$ and $G_{2}$. In order to do this, we will compute the time evolution of the relative entropy by using lemma 4.4 of \cite{lacker}:
\begin{align} \label{ineq_ent}
 H_{t}^{u}-H_{t}^{0} = C \int_{0}^{t} \int P_{s}^{1,u} \norm{\int_{0}^{1} G_{1}(u,v) \left \langle b(x_{u},\cdot),P_{s}^{1,v} \right \rangle dv-\int_{0}^{1} G_{2}(u,v) \left \langle b(x_{u},\cdot),P_{s}^{2,v} \right \rangle dv}^{2}.
\end{align}
Note that we can use this Lemma because the SDE satisfied by $X_{u}^{2} $ is well posed according to definition 2.1 of \cite{lacker} (see \cite{nonlinear_graphon} for instance)
. In a very similar fashion to the previous section, we will obtain a system of differential inequalities and prove the result by using a Gronwall type argument. In order to use this Gronwall argument we will rely on an a priori estimate on $H_{t}^{u}$, namely that it is bounded by a constant uniformly in $u$. Notice that, because $b$ is bounded, there exists $C>0$ such that 
\begin{align*}
\forall v \in [0,1], \forall i \in \{1,2\}, \forall x_{u} \in \mathbb{R}^{d}, \left \langle b(x_{u}, \cdot), P_{t}^{i,v} \right \rangle \leq C.
\end{align*}
Because $G_{1}$ and $G_{2}$ take values in $[0,1]$, inequality \eqref{ineq_ent} implies the boundedness of $H_{t}^{u}$ uniformly in $u$. We now use \eqref{ineq_ent} to obtain an inequality on $H_{t}^{u}$. In order to do that, let us denote, for $s \in [0,t]$,
\[A_{s}=\int P_{s}^{1,u} \norm{\int_{0}^{1} G_{1}(u,v) \left \langle b(x_{u},\cdot),P_{s}^{1,v} \right \rangle dv-\int_{0}^{1} G_{2}(u,v) \left \langle b(x_{u},\cdot),P_{s}^{2,v} \right \rangle dv}^{2}.\]
We now bound the term $A_{s}$ for all s.
Notice that we have, by a triangular inequality
\begin{align*}
A_{s}& \leq C  \int P_{s}^{1,u} \norm{ \int_{0}^{1} \left( G_{1}(u,v)-G_{2}(u,v)\right) \langle b(x_{u},\cdot), P_{s}^{1,v} \rangle dv}^{2} 
\\ &
\qquad + C \int P_{s}^{1,u} \norm{\int_{0}^{1} G_{2}(u,v) \langle b(x_{u},\cdot), P_{s}^{1,v}-P_{s}^{2,v} \rangle dv}^{2} 
\\ &
=A_{1,s}+A_{2,s}.
\end{align*}
We can bound the first term by using the fact that b is bounded. We have 
\begin{align*}
A_{1,s} & = C \int P_{s}^{1,u} \norm{ \int_{0}^{1} \left( G_{1}(u,v)-G_{2}(u,v)\right) \langle b(x_{u},\cdot), P_{s}^{1,v} \rangle dv }^{2}
\\ &
\leq  C \int P_{s}^{1,u} \left( \int_{0}^{1} \left \lvert G_{1}(u,v)-G_{2}(u,v) \right \rvert \norm{ \langle b(x_{u},\cdot), P_{s}^{1,v} \rangle } dv \right)^{2}
\\ &
\leq C  \int P_{s}^{1,u} \left( \int_{0}^{1} \lvert G_{1}(u,v)-G_{2}(u,v) \rvert dv \right)^{2}
\\ &
\leq C d(G_{1},G_{2})^{2}.
\end{align*}
Let us now bound the second term. We have
\begin{align*}
A_{2,s}&= C \int P_{s}^{1,u} \norm{ \int_{0}^{1} G_{2}(u,v) \langle b(x_{u},\cdot), P_{s}^{1,v}-P_{s}^{2,v} \rangle dv }^{2}
\\ &
\leq C \int P_{s}^{1,u} \left( \int_{0}^{1} G_{2}(u,v) \norm{ \langle b(x_{u},\cdot), P_{s}^{1,v}-P_{s}^{2,v} \rangle} dv \right)^{2}.
\end{align*}
Notice moreover that we have
\begin{align*}
\left \langle b(x_{u},\cdot), P_{t}^{i,v} \right \rangle& =\mathbb{E} \left(b(x_{u}, X_{t}^{v}) \right)
\\ &
=\left \langle b(x_{u},\cdot), P_{[t]}^{i,v} \right \rangle.
\end{align*}
By Pinsker's inequality, we have:
\begin{align*}
\lVert \langle b(x_{u},\cdot), P_{[t]}^{1,v}-P_{[t]}^{2,v} \rangle \rVert \leq \lVert b \rVert_{\infty}^{2} \sqrt{ H(P_{[t]}^{1,v}|P_{[t]}^{2,v})}.
\end{align*}
Therefore, we have, by the Cauchy Schwarz inequality (on the first line) and using the fact that $G_{2}$ is bounded (on the second line),
\begin{align*}
A_{2,s} &\leq C  \int P_{s}^{1,u} \left( \int_{0}^{1} G_{2}(u,v) dv \right) \int_{0}^{1} G_{2}(u,v) \norm{ \langle b(x_{u},\cdot), P_{s}^{1,v}-P_{s}^{2,v} \rangle}^{2} dv 
\\ &
\leq C  \int P_{s}^{1,u} \int_{0}^{1} G_{2}(u,v) \norm{\langle b(x_{u},\cdot), P_{[s]}^{1,v}-P_{[s]}^{2,v} \rangle}^{2} dv 
\\ &
\leq C  \int P_{s}^{1,u} \int_{0}^{1}G_{2}(u,v) H_{s}^{v} dv
\\ &
\leq C \int_{0}^{1} G_{2}(u,v) H_{s}^{v}dv.
\end{align*} 
We therefore get the following bound for the relative entropy 
\begin{align} \label{ineg_diff_ent}
H_{t}^{u} &\leq H_{t}^{0}+ C t d(G_{1},G_{2})^{2}+C \int_{0}^{t}\int_{0}^{1} G(u,v) H_{s}^{v}dv ds
\\ &
\leq H_{0}^{u}+C t d(G_{1},G_{2})^{2}+ \int_{0}^{t} \mathcal{A}H_{s}^{u} ds, 
\end{align}
where $\mathcal{A}$ is an operator from $\mathcal{L}^{\infty}([0,1])$ to $\mathcal{L}^{\infty}([0,1])$ defined as
\[\forall f \in \mathcal{L}^{\infty}([0,1]), \forall u \in [0,1], \mathcal{A} f(u)=\int_{0}^{1} G(u,v) f(v) dv,\]
where $H_{t}$ is seen as an element of $\mathcal{L}^{\infty}([0,1])$ such that $H_{t}(u):=H_{t}^{u}$. 
We now let 
\[F_{t}(u)=\int_{0}^{t} \mathcal{A} H_{s}^{u} ds.\]
Notice that \eqref{ineq_ent} implies the continuity of $t\to H_{t}^{v}$ for all $v$ so that, by the  boundedness of $H_{t}^{v}$ uniformly in $v$, $t\to \mathcal{A} H_{t}^{u}$ is continuous. We therefore have 
\[\forall t \in \mathbb{R}, \ \frac{d}{dt} F_{t}^{u}=\mathcal{A} H_{s}^{u}.\]
Notice moreover that, if $F \leq G$ then $\mathcal{A} F \leq \mathcal{A} G$ so that
\begin{align*}\frac{d}{dt} F_{t}^{u} & \leq \mathcal{A} H_{0}^{u}+Ct d(G_{1},G_{2})^{2} \mathcal{A} 1+\mathcal{A} F_{t}^{u}
\\ &
\leq \mathcal{A} H_{0}^{u}+Ct d(G_{1},G_{2})^{2}+\mathcal{A} F_{t}^{u},
\end{align*}
where $1$ denotes the function that is constant, equal to 1.
Let us now prove that, for all $t \geq 0$, $e^{t \mathcal{A}}$ is a positive operator. We have, if $G$ is a positive function:
\[\forall u \in \mathcal{P}_{f}([0,1]), \mathcal{A} G^{u}=C\int_{0}^{1} G(u,v) G(v) dv \geq 0.\]
It is then easy to see, by induction, that, for all $n \in \mathbb{N}$
\[\mathcal{A}^{n} G^{u} \geq 0.\]
Therefore, for all $t \geq 0$, 
\[e^{t \mathcal{A}} G^{u}=\sum_{n \in \mathbb{N}} \frac{t^{n}}{n!} \mathcal{A}^{n} G^{u} \geq 0.\]
Using this property, it is possible to use Gronwall's Lemma in the following way, because $e^{t \mathcal{A}}$ preserves the order 
\begin{align*}
\frac{d}{dt} e^{t \mathcal{A}} F_{T-t}& =e^{t \mathcal{A}} (\mathcal{A} F_{T-t}-\frac{d}{dt}F_{T-t})
\\ &
\geq -C t\ d(G_{1},G_{2})^{2}  e^{t \mathcal{A}} 1 -e^{t \mathcal{A}} \mathcal{A} H_{0}.
\end{align*} 
By integrating this inequality between the times $0$ and $T$, we get, for all $u \in [0,1]$, because $F_{0}=0$
\begin{align*}
F_{T}^{u} \leq C d(G_{1},G_{2})^{2} \int_{0}^{T} t e^{t \mathcal{A}} 1 dt+\int_{0}^{T} e^{t \mathcal{A}} \mathcal{A} H_{0} dt.
\end{align*}
We now bound the term $e^{t \mathcal{A}} 1$:
We have
\begin{align*}
\mathcal{A} 1^{u}=C \int_{0}^{1} G(u,v) dv \leq C.
\end{align*}
Therefore, by induction, it is possible to show that, for all $n \in \mathbb{N}$
 \begin{align*}
\mathcal{A}^{n} 1^{u} \leq C^{n}.
 \end{align*}
By using the Taylor expansion of the exponential operator, we get
\[\forall t \geq 0, \forall u \in [0,1], e^{t \mathcal{A}} 1^{u} \leq e^{C t}.\]
Plugging this bound on the one we have for $F_{T}^{u}$, and using the assumption on $H_{0}^{u}$ yields
\begin{align*}
F_{T}(u) \leq C T e^{C T} d(G_{1},G_{2})^{2}.
\end{align*}
Finally, we combine the previous estimate with the bound \eqref{ineg_diff_ent} to get the final bound 
\[H_{T}^{u} \leq C T e^{C T} d(G_{1},G_{2})^{2}.\]
\section{Stability of graphon mean field systems in Fisher information} \label{stabilite}
Let us now focus on proving bounds on Fisher information. The proof will be very similar to the previous sections. We can use Lemma 2.4 from \\cite{fisher}, all the integration by parts being justified, while $\nabla^{2} Q_{t}^{u}$ is bounded by a constant that doesn't depend on $u$ according to Lemma \ref{lemme_hessienne}. Therefore, we have, for all $u \in [0,1]$:
\begin{align*}
\frac{d}{dt} I_{t}^{u} & \leq C I_{t}^{u}+\ \int P_{t}^{1,u} \norm{\int_{0}^{1}  G_{1}(u,v) \langle b(x_{u},\cdot),P_{t}^{1,v} \rangle dv-\int_{0}^{1}  G_{2}(u,v) \langle b(x_{u},\cdot),P_{t}^{2,v} \rangle dv}^{2}
\\ &
+ \int P_{t}^{1,u} \norm{\nabla_{x_{u}} \left(\int_{0}^{1}  G_{1}(u,v) \langle b(x_{u},\cdot),P_{t}^{1,v} \rangle dv-\int_{0}^{1}  G_{2}(u,v) \langle b(x_{u},\cdot),P_{t}^{2,v} \rangle dv \right)}^{2}.
\end{align*}
Notice moreover that
\begin{align*}
\nabla_{x_{u}} \left(\int_{0}^{1}  G_{1}(u,v) \langle b(x_{u},\cdot),P_{t}^{1,v} \rangle dv\right)=\int_{0}^{1}  G_{1}(u,v) \langle \nabla b(x_{u},\cdot),P_{t}^{1,v} \rangle dv.
\end{align*}
Finally, we have
\begin{align*}
\frac{d}{dt} I_{t}^{u} & \leq C I_{t}^{u}+\int P_{t}^{1,u} \norm{\int_{0}^{1}  G_{1}(u,v) \langle b(x_{u},\cdot),P_{t}^{1,v} \rangle dv-\int_{0}^{1}  G_{2}(u,v) \langle b(x_{u},\cdot),P_{t}^{2,v} \rangle dv}^{2}
\\ &
+ \int P_{t}^{1,u} \norm{\int_{0}^{1}  G_{1}(u,v) \langle \nabla b(x_{u},\cdot),P_{t}^{1,v} \rangle dv-\int_{0}^{1}  G_{2}(u,v) \langle \nabla b(x_{u},\cdot),P_{t}^{2,v} \rangle dv}^{2}.
\end{align*}
By using very similar arguments as those of section 2, we have, because $b$ and $\nabla b$ are bounded
\begin{align} \label{eq_diff_fisher}
\frac{d}{dt} I_{t}^{u} \leq C I_{t}^{u}+\mathcal{A} H_{t}^{u}+C d(G_{1},G_{2})^{2}.
\end{align}
We also proved in section \ref{sec_4} that 
\begin{align*}
H_{t}^{u} \leq C t e^{Ct} d(G_{1},G_{2})^{2},
\end{align*}
so that we have, because $\mathcal{A} 1$ is bounded 
\[\mathcal{A} H_{t}^{u} \leq C t e^{Ct} d(G_{1},G_{2})^{2}.\]
Plugging this bound in \ref{eq_diff_fisher} yields
\begin{align*}
\frac{d}{dt} I_{t}^{u} \leq C I_{t}^{u}+C t e^{Ct}d(G_{1},G_{2})^{2}.
\end{align*}
Using Gronwall's lemma yields the existence of $C$ that does not depend on $u$ or $t$ such that
\[\forall t \in \mathbb{R}^{+}, \forall u \in [0,1], I_{t}^{u} \leq Ct e^{Ct} d(G_{1},G_{2})^{2}.\]
This concludes the proof.
\section{Strong existence of solutions to PDEs} \label{sol_forte}
Let us prove that, if ($\left(X_{u} \right)_{u \in [0,1]}$ is a solution of the equation, where $G$ is a bounded graphon
\begin{align*}
\left\{
            \begin{array}{ll}
              dX_{u}(t)=\int_{0}^{1} \int_{\mathbb{R}^{d}} b(X_{u}(s),x) G(u,v) P_{t}^{v}(dx)dvdt+  dB_{u}(t), \\ 
              P_{t}^{v}=\text{Law}(X_{u}(t)),
            \end{array}
            \right.
\end{align*}
then $P_{t}^{u}$ is a strong solution of the equation 
\begin{align*}
\partial_{t} P_{t}^{u}=  \Delta_{x_{u}} P_{t}^{u}-\nabla_{x_{u}} \cdot \left(\int_{0}^{1}  G(u,v) \langle b(x_{u}-\cdot), P_{t}^{v} \rangle dv \ P_{t}^{u} \right).
\end{align*}
Existence and uniqueness for the graphon mean field system has be done for instance in \cite{nonlinear_graphon}. While the regularity in $u$ of the system poses difficulties, if $u$ is fixed, the particle $X_u$ satisfies the above equation in the conventional sense of stochastic differential equations, according to the same paper. Since we will use it later, we will apply ito's formula to $\phi_{t} (X_{t}^{u})$, where $(s,x)\to \phi(s,x)$ is smooth, bounded with bounded derivatives
\begin{align*}
\left \langle \phi_{t}, P_{t}^{u} \right \rangle & = \int_{0}^{t} \left \langle \Delta \phi_{s}, P_{s}^{u} \right \rangle+\left \langle \nabla \phi_{s}, P_{s}^{u} \int_{0}^{1} G(u,v) \ b \ast P_{s}^{v} dv \right \rangle ds +\mathbb{E} \left(\int_{0}^{t}   \ \nabla \phi(X_{s}^{u}) dB_{s}^{u} \right)
\\ & \qquad +\int_{0}^{t} \left \langle \partial_{s} \phi_{s}, P_{s}^{u} \right \rangle ds
\\ &
=\int_{0}^{t}  \left \langle \Delta \phi_{s}, P_{s}^{u} \right \rangle+\left \langle \nabla \phi_{s}, P_{s}^{u} \int_{0}^{1} G(u,v) \ b \ast P_{s}^{v} dv \right \rangle ds +\int_{0}^{t} \left \langle \partial_{s} \phi_{s}, P_{s}^{u} \right \rangle ds.
\end{align*}
 Notice moreover that $x_{u} \to \int_{0}^{1} G(u,v) \ b \ast P_{t}^{v} dv$ is smooth. By denoting this function $b_{t}^{u}$, we therefore know that $P_{t}^{u}$ satisfies the following equation in a weak sense
\begin{align} \label{eq_faible}
\partial_{t} f_{t}=  \Delta f_{t}-\nabla \cdot \left(b_{t}^{u}  f_{t} \right).
\end{align}
Notice that, because $t \to b_{t}^{u}$ and $t \to \nabla \cdot b_{t}^{u}$ are continuous (by the dominated convergence theorem, because $t \to P_{t}^{v}$ is continuous for all v), if $b^{u}$ is fixed, there exists a strong solution $f_{t}$ to the equation above, such that $P_{0}^{u}=f_{0}$ (according for instance to theorem 11 of chapter 2 of \cite{friedman1983partial}). We will prove that $P_{t}^{u}=f_{t}$ for all $t$. Let us consider, for $0 \leq s \leq t$ and $y \in \mathbb{R}^{d}$ the equation
\begin{equation*}
dY_{s}^{t}=b_{t}^{u}(Y_{s}^{t})dt+  dB_{t}, \qquad Y_{s}^{s}=y.
\end{equation*}
If $G_{s}^{t}(y)$ is the law of $Y_{s}^{t}(y)$, let us denote $G_{s}^{t} \phi(y)=\left \langle G_{s}^{t}(y), \phi \right \rangle $. It is known that it satisfies the Backward Kolmogorov equation, namely that, for all $\phi \in \mathcal{C}_{c}^{\infty}$
\begin{align*}
\partial_{s} G_{s}^{t} \phi=-b_{s}^{u} \nabla G_{s}^{t} \phi-  \Delta G_{s}^{t} \phi.
\end{align*}
Therefore, we have, because $G_{s}^{t} \phi$ is smooth, bounded with bounded derivatives (see Theorem II.4.4 of \cite{Kunita1984} for a proof that the flow of a stochastic differential equation with $C^{k}$ coefficients is $C^{k}$ in space)
\begin{align*}
 \left \langle G_{s}^{t} \phi, P_{s}^{u} \right \rangle & = \left \langle G_{0}^{t} \phi, P_{0}^{u} \right \rangle + \int_{0}^{s} -\left \langle b_{v}^{u} \nabla G_{v}^{t} \phi+  \Delta G_{v}^{t} \phi, P_{v}^{u} \right \rangle dv + \int_{0}^{s} \left \langle   \Delta G_{v}^{t} \phi+b_{v}^{u} \nabla \phi, P_{v}^{u} \right \rangle dv \\ &
 =\left \langle G_{0}^{t} \phi, P_{0}^{u} \right \rangle.
\end{align*}
Taking $s=t$ yields, because $G_{t}^{t} \phi=\phi$
\begin{align*}
\left \langle \phi, P_{t}^{u} \right \rangle =\left \langle G_{0}^{t} \phi, P_{0}^{u} \right \rangle.
\end{align*}
Because $f_{t}$ is a weak solution of \eqref{eq_faible}, by using very similar arguments as before, we can prove that 
\begin{align*}
\left \langle \phi, f_{t} \right \rangle =\left \langle G_{0}^{t} \phi, f_{0} \right \rangle.
\end{align*}
Because $f_{0}=P_{0}^{u}$, we get that $\left \langle \phi, f_{t}-P_{t}^{u} \right \rangle=0 $ for all $\phi$, so that $f_{t}=P_{t}^{u}$ for all t.

\section{Bound on the hessian} \label{sec_7}
\begin{lemma} \label{lemme_hessienne}
Suppose that there exists $V$ such that $b=\nabla V$, and that $b$ and its derivatives are bounded. Suppose moreover that there exists $C_{0}$ such that 
\[\forall u \in [0,1], \forall x \in \mathbb{R}^{d}, \norm{\nabla^{2} \log P_{0}^{2,u}(x)} \leq C_{0}.\]
Then, there exists $C>0$ such that
\[\forall u \in [0,1], \forall x \in \mathbb{R}^{d}, \norm{\nabla^{2} \log P_{t}^{2,u}(x)} \leq C .\]
\end{lemma}
\begin{proof}
Under our assumptions, we have
\begin{align*}
\partial_{t} P_{t}^{2,u}& =  \Delta_{x_{u}} P_{t}^{2,u}- \nabla_{x_{u}} \cdot \left(\int_{0}^{1}  G_{2}(u,v) \langle b(x_{u},\cdot),P_{t}^{2,v} \rangle dv P_{t}^{2,u} \right)
\\ &
=  \Delta_{x_{u}} P_{t}^{2,u}- \nabla_{x_{u}} \cdot \left(b_{t}^{u} P_{t}^{2,u} \right).
\end{align*}
Where $b_{t}^{2,u}$ is smooth, uniformly bounded in u, with derivatives uniformly bounded in u. Because we assumed that $b=\nabla U$ for some U, there exists $U_{t}^{u}$ such that $b_{t}^{u}=\nabla U_{t}^{u}$ for all $t$ and $u$. 
Therefore, $P_{t}^{2,u}$ satisfies the equation
\begin{align*}
\partial_{t} P_{t}^{u}=  \Delta_{x_{u}} P_{t}^{u}-\nabla \cdot \left(b_{t}^{u} \right) P_{t}^{2,u}-\nabla U_{t}^{u} \cdot \nabla P_{t}^{2,u}.
\end{align*}
This means that $P_{t}^{2,u}$ satisfies the same equation as $\psi_{t}$ in section 2 of \cite{bdd_hess}, so $-\log P_{T-t}^{2,u}$ satisfies the same equation as $\phi_{t}$ with $V_{t}^{u} \coloneqq \nabla \cdot b_{t}^{2,u}$. Using the same notations as \cite{bdd_hess}, we have $h_{u}=-\log P_{0}^{2,u}$, so that $\nabla^{2} h_{u}$ is uniformly bounded in u (and thus $\nabla h_{u}$ is uniformly lipschitz, with a lipschitz constant that does not depend on u). We also know that $\nabla U_{t}^{u}$ is uniformly lipschitz in u. Notice moreover that, because our coefficients may depend on time, we have to consider a generalized version of the operator $\mathcal{A}_{U}$, so we need to check that $\nabla \left[\mathcal{A}_{U}+V_{t}^{u}\right]=\nabla \left[ \frac{1}{2} | \nabla U_{t}{u} |^{2}-\frac{1}{2} \Delta U_{t}^{u}+V_{t}^{u} \right]-\nabla \left(\partial_{t} U_{t}^{u} \right)$ is uniformly lipschitz in u. Because $b_{t}^{u}$ and its derivatives are uniformly bounded in u, $\nabla \left[ \frac{1}{2} | \nabla U_{t}^{u} |^{2}-\frac{1}{2} \Delta U_{t}^{u}+V_{t}^{u} \right]$ is uniformly lipschitz in u. Moreover, we have, by integration by parts
\begin{align*} \nabla \left(\partial_{t} U_{t}^{u} \right) & =\partial_{t} \nabla U_{t}^{u}
\\ &
=\partial_{t} \int_{0}^{1} G_{2}(u,v) \left \langle b(x_{u}-\cdot), P_{t}^{2,v} \right \rangle dv
\\ &
= \int_{0}^{1} G_{2}(u,v) \left \langle b(x_{u}-\cdot), \Delta P_{t}^{2,v} \right \rangle dv - \int_{0}^{1} G_{2}(u,v) \left \langle b(x_{u}-\cdot), \nabla \cdot \left(b_{t}^{v} P_{t}^{v} \right) \right \rangle dv
\\ &
= \int_{0}^{1} G_{2}(u,v) \left \langle \Delta b(x_{u}-\cdot),  P_{t}^{2,v} \right \rangle dv + \int_{0}^{1} G_{2}(u,v) \left \langle \nabla b(x_{u}-\cdot), b_{t}^{v} P_{t}^{v} \right \rangle dv.
\end{align*}
Because $b$ and its derivatives are bounded, this identity entails that $\nabla \partial_{t} U_{t}^{u}$ is uniformly lipschitz in $x_{u}$, with a constant that does not depend on $u$.
We also have, because $\mathcal{A}+V_{t}^{u}$ is bounded, that  \[\kappa_{\mathcal{A}_{t}+V_{t}^{u}}=\frac{1}{r^{2}} \sup_{|x-y|=r} \langle \nabla \mathcal{A}_{t}(x)+V_{t}^{u}(x)-\mathcal{A}_{t}(y)-V_{t}^{u}(y),x-y\rangle \geq -\frac{C}{r}. \]
According to the first point of Section 2.2, we can take $f_{t}(r)=-\frac{r}{t+\alpha}+\frac{\beta_{1}}{2} (t+\alpha)-\frac{\beta_{2}}{t+\alpha}$ for a suitable choice of $\beta_{1}, \beta_{2}$ and $\alpha$ with which we can apply Corollay 2.5 of \cite{bdd_hess} with $ \sigma=\text{Id}$. This entails that $\nabla^{2} \log P_{T-t}^{u}$ is bounded for all $t,T$. Moreover, the bound proven in Corollary 2.5 only depends on the lipschitz constants of $\nabla U_{t}^{u}$, $\nabla V_{t}^{u}$ and $\nabla h^{u}$. The bound we get on $\nabla^{2} \log P_{t}^{2,u}$ is therefore uniform in u. This concludes the proof of the Lemma.
\end{proof}

\medskip

\noindent{\bf Acknowledgments.}\\
This work has been (partially) supported by the Project CONVIVIALITY ANR-23-CE40-0003 of the French National Research Agency. The author would also like to thank his PhD advisors Christophe Poquet and Arnaud Guillin for their suggestions on this present paper.

\bibliographystyle{abbrv}
\bibliography{biblio.bib}

\begin{thebibliography}{10}

\bibitem{bayraktar2022graphon}
E.~Bayraktar, S.~Chakraborty, and R.~Wu.
\newblock Graphon mean field systems.
\newblock {\em Ann. Appl. Probab.}, 33(5):3587--3619, 2023.

\bibitem{Bayraktar_2024}
E.~Bayraktar and D.~Kim.
\newblock Concentration of measure for graphon particle system.
\newblock {\em Advances in Applied Probability}, 56(4):1279–1306, Jan. 2024.

\bibitem{bayraktar2021graphon}
E.~Bayraktar and R.~Wu.
\newblock Stationarity and uniform in time convergence for the graphon particle
  system.
\newblock {\em Stochastic Processes Appl.}, 150:532--568, 2022.

\bibitem{Bet_2023}
G.~Bet, F.~Coppini, and F.~R. Nardi.
\newblock Weakly interacting oscillators on dense random graphs.
\newblock {\em Journal of Applied Probability}, 61(1):255–278, June 2023.

\bibitem{BHAMIDI20192174}
S.~Bhamidi, A.~Budhiraja, and R.~Wu.
\newblock Weakly interacting particle systems on inhomogeneous random graphs.
\newblock {\em Stochastic Processes and their Applications}, 129(6):2174--2206,
  2019.

\bibitem{Budhiraja2015}
A.~Budhiraja and R.~Wu.
\newblock Some fluctuation results for weakly interacting multi-type particle
  systems.
\newblock {\em Stochastic Processes Appl.}, 126(8):2253--2296, 2016.

\bibitem{bdd_hess}
L.-P. Chaintron, G.~Conforti, and K.~Eichinger.
\newblock Propagation of weak log-concavity along generalised heat flows via
  hamilton-jacobi equations, 2025.

\bibitem{Chaintron_2}
L.-P. Chaintron and A.~Diez.
\newblock Propagation of chaos: A review of models, methods and applications.
  i. models and methods.
\newblock {\em Kinetic and Related Models}, 15(6):895, 2022.

\bibitem{Chaintron_1}
L.-P. Chaintron and A.~Diez.
\newblock Propagation of chaos: A review of models, methods and applications.
  ii. applications.
\newblock {\em Kinetic and Related Models}, 15(6):1017, 2022.

\bibitem{Collet2014}
F.~Collet.
\newblock Macroscopic limit of a bipartite curie–weiss model: A dynamical
  approach.
\newblock {\em Journal of Statistical Physics}, 157(6):1301–1319, Sept. 2014.

\bibitem{Contucci2008}
P.~Contucci, I.~Gallo, and G.~Menconi.
\newblock Phase transitions in social science: two-population mean field
  theory.
\newblock {\em International Journal of Modern Physics B}, 22(14):2199–2212,
  June 2008.

\bibitem{nonlinear_graphon}
F.~Coppini, A.~De~Crescenzo, and H.~Pham.
\newblock Nonlinear graphon mean-field systems.
\newblock {\em Stochastic Processes Appl.}, 190:19, 2025.
\newblock Id/No 104728.

\bibitem{Delattre2016}
S.~Delattre, G.~Giacomin, and E.~Luçon.
\newblock A note on dynamical models on random graphs and fokker–planck
  equations.
\newblock {\em Journal of Statistical Physics}, 165(4):785–798, Nov. 2016.

\bibitem{friedman1983partial}
A.~Friedman.
\newblock {\em Partial Differential Equations of Parabolic Type}.
\newblock R.E. Krieger Publishing Company, 1983.

\bibitem{fisher}
J.~Grass, A.~Guillin, and C.~Poquet.
\newblock Propagation of chaos in fisher information, 2025.

\bibitem{article_diffusion}
J.~Grass, A.~Guillin, and C.~Poquet.
\newblock Sharp propagation of chaos for mckean-vlasov equation with non
  constant diffusion coefficient.
\newblock {\em Electronic Communications in Probability}, 30:1--12, 2025.

\bibitem{Jabin_2024}
P.~Jabin, D.~Poyato, and J.~Soler.
\newblock Mean‐field limit of non‐exchangeable systems.
\newblock {\em Communications on Pure and Applied Mathematics},
  78(4):651–741, Nov. 2024.

\bibitem{Jabin_Wang}
P.-E. Jabin and Z.~Wang.
\newblock Quantitative estimates of propagation of chaos for stochastic systems
  with ${W}^{-1,\infty }$ kernels.
\newblock {\em Inventiones mathematicae}, 214(1):523–591, July 2018.

\bibitem{Kunita1984}
H.~Kunita.
\newblock Stochastic differential equations and stochastic flows of
  diffeomorphisms.
\newblock In P.~L. Hennequin, editor, {\em {\'E}cole d'{\'E}t{\'e} de
  Probabilit{\'e}s de Saint-Flour XII - 1982}, pages 143--303, Berlin,
  Heidelberg, 1984. Springer Berlin Heidelberg.

\bibitem{lacker}
D.~Lacker.
\newblock Hierarchies, entropy, and quantitative propagation of chaos for mean
  field diffusions.
\newblock {\em Probability and Mathematical Physics}, 4(2):377--432, May 2023.

\bibitem{lacker_uniforme}
D.~Lacker and L.~Le~Flem.
\newblock Sharp uniform-in-time propagation of chaos.
\newblock {\em Probability Theory and Related Fields}, 187(1-2):443--480, 2023.

\bibitem{lacker_graphe}
D.~Lacker, L.~C. Yeung, and F.~Zhou.
\newblock Quantitative propagation of chaos for non-exchangeable diffusions via
  first-passage percolation, 2024.

\bibitem{lovasz2012}
L.~Lov{\'a}sz.
\newblock {\em Large Networks and Graph Limits}.
\newblock American Mathematical Society colloquium publications. American
  Mathematical Society, 2012.

\bibitem{lovasz2004}
L.~Lov{\'a}sz and B.~Szegedy.
\newblock Limits of dense graph sequences.
\newblock {\em J. Comb. Theory, Ser. B}, 96(6):933--957, 2006.

\bibitem{Muntean2014}
A.~Muntean and F.~Toschi.
\newblock {\em Collective Dynamics from Bacteria to Crowds: An Excursion
  Through Modeling, Analysis and Simulation}, volume 553.
\newblock 01 2014.

\bibitem{Nadtochiy2019}
S.~Nadtochiy and M.~Shkolnikov.
\newblock Mean field systems on networks, with singular interaction through
  hitting times.
\newblock {\em Ann. Probab.}, 48(3):1520--1556, 2020.

\bibitem{Naldi2010}
G.~Naldi, L.~Pareschi, and G.~Toscani, editors.
\newblock {\em Mathematical modeling of collective behavior in socio-economic
  and life sciences}.
\newblock Model. Simul. Sci. Eng. Technol. Boston, MA: Birkh{\"a}user, 2010.

\bibitem{Oliveira2019}
R.~I. Oliveira and G.~H. Reis.
\newblock Interacting diffusions on random graphs with diverging average
  degrees: Hydrodynamics and large deviations.
\newblock {\em Journal of Statistical Physics}, 176(5):1057–1087, July 2019.

\bibitem{Sirignano2019}
J.~Sirignano and K.~Spiliopoulos.
\newblock Mean field analysis of neural networks: a law of large numbers.
\newblock {\em SIAM J. Appl. Math.}, 80(2):725--752, 2020.

\bibitem{Vil25}
C.~Villani.
\newblock Fisher information in kinetic theory, 2025.
\newblock https://arxiv.org/abs/2501.00925.

\bibitem{Vortex}
S.~Wang.
\newblock Sharp local propagation of chaos for mean field particles with
  $w^{-1,\infty}$ kernels.
\newblock {\em Journal of Functional Analysis}, page 111240, Oct. 2025.

\end{thebibliography}

\end{document}